\documentclass[12pt]{amsart}
\usepackage{a4}
\usepackage{booktabs}
\usepackage[T1]{fontenc}
\usepackage[all]{xy}
\usepackage{lipsum}
\usepackage{url}
\usepackage{tikz-network}
\usepackage{stackrel}
\usepackage{color}

\usepackage{amsmath,amsthm,amssymb}

\newcommand{\aspas}[1]{``{#1}''}
\usepackage{xcolor}
\usepackage{graphicx}

\usepackage{tikz}
\usetikzlibrary{
  hobby,
  intersections,
  spath3,
  decorations.markings,
  arrows.meta,
}

\newtheorem{theorem}{Theorem}[section]
\newtheorem{lemma}[theorem]{Lemma}
\newtheorem{example}[theorem]{Example}
\newtheorem{proposition}[theorem]{Proposition}
\theoremstyle{definition}
\newtheorem{definition}[theorem]{Definition}

\newtheorem{remark}[theorem]{Remark}
\newtheorem{corollary}[theorem]{Corollary}
\numberwithin{equation}{section}

\begin{document}


\renewcommand{\bf}{\bfseries}
\renewcommand{\sc}{\scshape}

\title[Equivariant category and Topological complexity of wedges]%
{Equivariant category and Topological complexity of wedges \\ }

\author{Cesar A. Ipanaque Zapata}
\address{Departamento de Matem\'atica, Universidade de S\~ao Paulo, Instituto de Matemá\-tica e Estatística, IME-USP, 
Rua do Mat\~ao 1010, CEP: 05508-090, S\~ao Paulo-SP, Brasil.} 
\email{cesarzapata@usp.br}
\thanks{The first author would like to thank grant \#2023/16525-7, \#2022/16695-7, \#2018/23678-6 and \#2016/18714-8, S\~ao Paulo Research Foundation (FAPESP)}

\author{Denise de Mattos}
\address{Departamento de Matem\'{a}tica, Universidade de São Paulo, Instituto de Ciências  
Matemáticas e de Computação, ICMC-USP, Avenida Traba\-lhador S\~{a}o-carlense, 400, Centro, CEP:
13566-590, S\~{a}o Carlos - SP, Brasil.}
\email{deniseml@icmc.usp.br}



\begin{abstract} We prove the formula \begin{equation*}
    \text{cat}_G(X\vee Y)=\max\{\text{cat}_G(X),\text{cat}_G(Y)\}
\end{equation*} for the equivariant category of the wedge $X\vee Y$. As a direct application, we have that the wedge $\bigvee_{i=1}^m X_i$ is $G$-contractible if and only if each $X_i$ is $G$-contractible, for each $i=1,\ldots,m$. One further application is to compute the equivariant category of the quotient $X/A$, for a $G$-space $X$ and an invariant subset $A$ such that the inclusion $A\hookrightarrow X$ is $G$-homotopic to a constant map $\overline{x_0}:A\to X$, for some $x_0\in X^G$. Additionally, we discuss the equivariant and invariant topological complexities for wedges. For instance, as applications of our results, we obtain the following equalities:  \begin{align*}
     \text{TC}_G(X\vee Y)&=\max\{\text{TC}_G(X),\text{TC}_G(Y),\text{cat}_G(X\times Y)\},\\ 
     \text{TC}^G(X\vee Y)&=\max\{\text{TC}^G(X),\text{TC}^G(Y),_{X\vee Y}\text{cat}_{G\times G}(X\times Y)\},
    \end{align*} for $G$-connected $G$-CW-complexes $X$ and $Y$ under certain conditions.       
\end{abstract}

\maketitle


\section{Introduction}
In this paper, \aspas{space} refers to a topological space, and by a \aspas{map} we always mean a continuous map. Fibrations are taken in the Hurewicz sense, meaning that they satisfy the homotopy lifting property for maps. Additionally, a $G$-fibration is one that satisfies the homotopy lifting property for $G$-maps.  

\medskip Let $G$ be a compact Hausdorff topological group acting continuously on a Hausdorff space $X$ on the left. The equivariant category of $X$, denoted by $\text{cat}_G(X)$, was introduced by Fadell in \cite{fadell1985}, as a generalization of the classical category, or Lusternik-Schnirelmann category (LS-category), of a space \cite{lyusternik1947topological}. For a gene\-ral overview of LS-category, we refer the reader to the survey article by James \cite{james1978category} and the book by Cornea-Lupton-Oprea-Tanré \cite{cornea2003lusternik}. In \cite[Proposition 1.27(2), p. 14]{cornea2003lusternik} the authors show that for path-connected normal spaces $X,Y$ with non-degenerate basepoints, the classical category of the wedge $X\vee Y$ is,
\begin{equation}\label{cat-wedge-nonequi}
   \mathrm{cat}(X\vee Y) = \max\{\mathrm{cat}(X), \mathrm{cat}(Y)\}, 
\end{equation} where $X\vee Y$ is the wedge of the disjoint sets $X$ and $Y$ obtained by identifying their basepoints.

\medskip Similar to the definition of classical category, $\mathrm{cat}_G(X)$ is defined as the least number of open invariant subsets of $X$ that form a covering for $X$, with each open subset being equivariantly contractible to an orbit, rather than a point. 

\medskip We note that a formula for the equivariant case of (\ref{cat-wedge-nonequi}) does not appear in the literature. In \cite[Lemma 2.7, p. 135]{bayeh2015}, the authors show that the inequality \begin{equation}
 \label{ineq}   \mathrm{cat}_G(X\vee Y)\leq \mathrm{cat}_G(X)+\mathrm{cat}_G(Y)-1
\end{equation} holds for any pointed $G$-spaces $(X,x_0)$ and $(Y,y_0)$. A pointed $G$-space means a $G$-space with a distinguished basepoint fixed by $G$.    

\medskip In this paper, we extend the classical formula~(\ref{cat-wedge-nonequi}) to the equivariant case, thus improving the inequality~(\ref{ineq}). Specifically, we show that under certain conditions (see Theorem \ref{principal}), the wedge $X\vee Y$ has equivariant category equal to \[\max\{\mathrm{cat}_G(X), \mathrm{cat}_G(Y)\}.\]

\medskip To this end, we extend the approach to LS-category given by Cornea-Lupton-Oprea-Tanré in \cite[Section 1.4]{cornea2003lusternik} to  the equivariant setting. 

Additionally, we discuss the \textit{equivariant} and \textit{invariant topological complexities} for wedge spaces. 

\medskip In simpler terms, the equivariant topological complexity $\mathrm{TC}_G$ is the minimal number of $G$-equivariant motion planning rules, whereas the invariant topological complexity $\mathrm{TC}^G$ is the minimal number of $G \times G$-equivariant motion planning rules required to navigate the space in a way that respects the group symmetry (see Definition~\ref{defn:equiv-inv-tc}).

\medskip The main results of this work are the following.
\begin{itemize}
\item We prove the formula $\text{cat}_G(X\vee Y)=\max\{\mathrm{cat}_G(X), \mathrm{cat}_G(Y)\}$ (Theorem \ref{principal}). 
    \item We show that the wedge $\bigvee_{i=1}^m X_i$ is $G$-contractible if and only if each $X_i$ is $G$-contractible, for each $i=1,\ldots,m$ (see Corollary~\ref{cor-wedge}). 
    \item We compute the equivariant category of the quotient $X/A$ for a $G$-space $X$ and an invariant subset $A$ such that the inclusion $A\hookrightarrow X$ is $G$-homotopic to a constant map $\overline{x_0}:A\to X$, for some $x_0\in X^G$ (Theorem~\ref{prop:equiv-quotient}).
    \item We compute the equivariant category $\text{cat}_G(X\wedge S^k)$ (Theorem~\ref{thm:equiv-smash}).
    \item We show that the inequalities \begin{align*}
       \max\{\text{TC}_G(X),\text{TC}_G(Y),\text{cat}_G(X\times Y)\}&\leq \text{TC}_G(X\vee Y) \quad \text{ and }\\ 
       \max\{\text{TC}^G(X),\text{TC}^G(Y),_{X\vee Y}\text{cat}_{G\times G}(X\times Y)\}&\leq \text{TC}^G(X\vee Y)
    \end{align*}  always hold (Theorem~\ref{thm:lower-bound-tc-wedge}).
    \item Under certain conditions, we obtain the following equalities: \begin{align*}
        \text{TC}_G(X\vee Y)&=\max\{\text{TC}_G(X),\text{TC}_G(Y),\text{cat}_G(X\times Y)\} \quad \text{ and }\\
        \text{TC}^G(X\vee Y)&=\max\{\text{TC}^G(X),\text{TC}^G(Y),_{X\vee Y}\text{cat}_{G\times G}(X\times Y)\} 
    \end{align*} (Theorem~\ref{prop:tc-2cat-1}).
\end{itemize} 

\medskip The paper is organized as follows: In Section~\ref{sec:pre}, we begin with a brief review of $G$-spaces, $G$-maps, $G$-homotopy, and $G$-invariant sets. Basic properties of inva\-riant sets are presented in Lemma~\ref{lem:basic-prop-inv}. We recall the definition of equivariant cate\-gory in Definition~\ref{defn:equiv-cat}. Some computation of equivariant category are shown in Example~\ref{exam:1-2}. Lemma~\ref{lem:g-contractible-connectivity} yields necessary conditions for $G$-contractiblity. The $G$-homotopy invariance of the equivariant category is presented in Proposition~\ref{prop-G-domina}. A key property of the equivariant category is presented in Lemma~\ref{conservation}. We introduce the notion of $y$-connectivity for a $G$-space in Definition~\ref{defn:y-connected}. As shown in Remark~\ref{rem:g-contractil-connected}, $G$-contractibility implies $y$-connectivity and, with an additional condition, $G$-connectivity. In Example~\ref{acao-hamil}, we use the $y$-connectivity property to compute the equivariant category of the standard Hamiltonian action. Proposition~\ref{o-cat-leq-equicat} shows that, under $y$-connectivity, the inclusion map of any $G$-categorical set is $G$-homotopic to a $G$-map with values in the orbit of $y$. We recall the definition of $G$-cofibration and $G$-retract in Definition~\ref{defn:g-cofibration}. Proposition~\ref{caracterizacao-Gcofibration} provides a characterization of $G$-cofibrations in terms of $G$-retracts, and thus we obtain a key property of $G$-cofibrations. We recall the notion of a $G$-well-pointed space in Definition~\ref{defn:g-well-pointed}. Proposition~\ref{G-ponto-base-nao-degenerado} shows that any $G$-non-degenerate basepoint admits a \aspas{good} open categorical neighborhood. In Definition~\ref{defn:g-normal}, we recall the notion of a $G$-normal space. Proposition~\ref{prop:normal-g-g-normal} shows that any normal $G$-space is a $G$-normal space. Proposition~\ref{primera-prop-G-normal} presents properties of $G$-normal spaces. 

In Section~\ref{sec:equiv-wedges}, we present our first main result (Theorem~\ref{principal}). For this purpose, in Proposition~\ref{G-categorico-baseado}, we show that for any $G$-connected, $G$-normal space with a $G$-non-degenerate basepoint, there exists a \aspas{good} open $G$-categorical cover. Such a \aspas{good} open $G$-categorical cover is called a based $G$-categorical covering (Definition~\ref{defn-based-g-cat}). The inequality $\max\{\text{cat}_G(X),\text{cat}_G(Y)\}\leq \text{cat}_G(X\vee Y)$ always holds, as shown in Proposition~\ref{prop:lower-bound-equiv-wedge}. Our first main theorem (Theorem~\ref{principal}) shows that this inequality is, in fact, an equality under certain conditions. Example~\ref{equiv-cat-x-x} shows that the equality $\text{cat}_G(X\vee X)=\text{cat}_G(X)$ holds. A characterization of the $G$-contractibility of the wedge $\bigvee_{i=1}^m X_i$ is shown in Corollary~\ref{cor-wedge}. The second main result in this section is Theorem~\ref{prop:equiv-quotient}, which computes the equivariant category of the quotient $X/A$ for a $G$-space $X$ and an invariant subset $A$ such that the inclusion $A\hookrightarrow X$ is $G$-homotopic to a constant map $\overline{x_0}:A\to X$, for some $x_0\in X^G$. Remark~\ref{rem:product-action} shows that we can consider more general product actions. Hence, we obtain Proposition~\ref{prop:equiv-cat-product} and Remark~\ref{rem:prod-lower-bound}. We end this section by discussing the equivariant category of the smash product (Remark~\ref{rem:smash-pro}). For instance, we compute the equivariant category $\text{cat}_G(X\wedge S^k)$ in Theorem~\ref{thm:equiv-smash}.

In Section~\ref{sec:equiv-inv-tc}, we discuss the equivariant and invariant topological complexity of a wedge. We recall the notions of equivariant and invariant topological complexities in Definition~\ref{defn:equiv-inv-tc}. In Lemma~\ref{lem:equi-inv-prop}, we record standard properties of these invariants. In Proposition~\ref{prop:quaotient-map-equivalence}, we compute the equivariant category, and the equivariant and invariant topological complexities of $X/A$ under the condition that $A$ is $G$-contractible to a fixed point $x_0\in X^G$. The first main theorem in this section is Theorem~\ref{thm:lower-bound-tc-wedge}, which shows that the inequalities \begin{align*}
  \max\{\text{TC}_G(X),\text{TC}_G(Y),\text{cat}_G(X\times Y)\}&\leq \text{TC}_G(X\vee Y) \quad \text{ and }\\
  \max\{\text{TC}^G(X),\text{TC}^G(Y),_{X\vee Y}\text{cat}_{G\times G}(X\times Y)\}&\leq \text{TC}^G(X\vee Y)
\end{align*}  always hold. Under certain conditions, we show that these inequalities are, in fact, equalities, as shown in Theorem~\ref{prop:tc-2cat-1}. For instance, in Proposition~\ref{prop:equiv-inv-tc-wedges-x-x}, we obtain that the equality \[\text{TC}_G(\underbrace{X\vee\cdots\vee X}_{k \text{ times }})=\text{TC}_G(X)\] holds, whenever $\text{TC}_G(X)=2\text{cat}_G(X)-1$. In addition, the equality \[\text{TC}^G(\underbrace{X\vee\cdots\vee X}_{k \text{ times }})=\text{TC}^G(X)\] holds, whenever $\text{TC}^G(X)=2\text{cat}_G(X)-1$. Several examples are presented to support these results (Example~\ref{exam:sphere}, Example~\ref{exam:x-sphere}, Example~\ref{exam:equiv-inv-x-a}, and Example~\ref{exam:equiv-inv-tc-smash}). We end this article with Remark~\ref{rem:conj}, which conjectures that the equalities \begin{align*}
    \text{TC}_G(X\vee Y)&=\max\{\text{TC}_G(X),\text{TC}_G(Y),\text{cat}_G(X\times Y)\},\\
    \text{TC}^G(X\vee Y)&=\max\{\text{TC}^G(X),\text{TC}^G(Y),_{X\vee Y}\text{cat}_{G\times G}(X\times Y)\},
\end{align*} always hold.  

\medskip The authors of this paper deeply thank the referees for very valuable comments and timely corrections on previous versions of the work.

\section{Preliminary results}\label{sec:pre}

In this section, we recall some definitions and fix our notation. We follow the standard conventions for compact transformation groups as used in \cite{bredon1976introduction} and \cite{tom1987transformation}.

Let $G$ be a compact Hausdorff topological group acting continuously on a Hausdorff space $X$ on the left. In this case, we say that $X$ is a \textit{$G$-space}.

For each $x \in X$, the {\emph{isotropy group}} (or {\emph{stabilizer}}) of $x$, denoted
\[
G_x := \{g \in G : gx = x\},
\]
is a closed subgroup of $G$.

The set
\[
{O}(x) := \{gx : g \in G\},
\]
also denoted by $Gx$, is called the {\emph{orbit}} of $x$.

\medskip There exists a homeomorphism from the quotient space $G/G_x$ to ${O}(x)$, which maps $gG_x$ to $g^{-1}x$, for each $g \in G$.

\medskip The \textit{orbit space} $X/G$ is the set of equivalence classes determined by the action, endowed with the quotient topology. Since $G$ is compact and $X$ is Hausdorff, the orbit space $X/G$ is also Hausdorff \cite[Proposition 3.1-(v), p. 5]{tom1987transformation}, and the \textit{quotient map} $p:X\to X/G$, sending each point to its orbit, is both open and closed \cite[Proposition 3.1, p. 22]{tom1987transformation}.

\medskip If $H$ is a closed subgroup of $G$, then the set 
\[
X^H:=\{x\in X:~hx=x, \text{ for all } h\in H\}
\]
is called the \textit{$H$-fixed point set} of $X$. In particular, the set $X^G$ is called the \textit{fixed point set} of $X$. Here, $X^H$ is endowed with the subspace topology. Note that $x\in X^G$ if and only if $O(x)=\{x\}$, which is equivalent to $G_x=G$. 

\medskip We denote the closed interval $[0,1]$ in $\mathbb{R}$ by $I$. Let $X$ and $Y$ be $G$-spaces. A \textit{$G$-equivariant map} or (a \textit{$G$-map}) $f:X\to Y$ is a map that satisfies $f(gx)=gf(x)$, for any $g\in G$ and $x\in X$. 

\medskip Let $X$ and $Y$ be $G$-spaces. Two $G$-equivariant maps $\phi,\psi:X\to Y$ are \textit{$G$-homotopic}, written $\phi\simeq_G\psi$, if there exists a $G$-map $F:X\times [0,1]\to Y$ such that $F_0=\phi$ and $F_1=\psi$, where $G$ acts trivially on $I=[0,1]$ (i.e., $gt=t$, for all  $g\in G$ and $t\in I$) and diagonally on $X\times [0,1]$ (i.e., $g(x,t)=(gx,gt)=(gx,t)$, for all $g\in G$ and $(x,t)\in X\times [0,1]$). A subset $U\subset X$ is described as \textit{$G$-invariant} (or simply \textit{invariant}) if $gU\subset U$, for all $g\in G$.

\medskip The following statement is straightforward to verify.

\begin{lemma}\label{lem:basic-prop-inv}\rm{
Let $X$ be a $G$-space. If $U\subset X$ is an invariant subset, then the closure $\overline{U}$ and the complements $ X-\overline{U}, X-U\subset X$ are also invariant. Furthermore, if $\{U_\alpha\}$ is a collection of invariant subsets of $X$, then the intersection $\bigcap U_\alpha$ and the complement $X-\bigcup_{\alpha}U_\alpha$ are also invariant.}
\end{lemma}

We recall the notion of the equivariant category of a $G$-space $X$, as studied, for instance, by Marzantowicz \cite{marzantowicz1989g} and  Collman-Grant \cite{colman2013equivariant}.

\medskip An invariant set $U$ in a $G$-space $X$ is called \textit{$G$-categorical} if the inclusion map $\text{incl}_U:U\hookrightarrow X$ is $G$-homotopic to a $G$-map $c:U\to X$ such that $c(U)\subset O(x)$, for some $x\in X$. That is, there exists a $G$-homotopy $F:U\times [0,1]\to X$ such that $F_0=\text{incl}_U$ and $F_1(U)\subset O(x)$, for some $x\in X$.

\medskip Now, we recall the definition of the equivariant category. 

\begin{definition}\label{defn:equiv-cat}
The \textit{equivariant category} of a $G$-space $X$, denoted $\text{cat}_G(X)$, is the least integer $k\geq 1$ such that $X$ can be covered by $k$ open sets $\{U_1,\ldots,U_k\}$, each of which is $G$-categorical. 
\end{definition}

Note that $\text{cat}_G(X)\geq \text{cat}(X/G)$ (using the fact that the quotient map $X\to X/G$ is open). The equality holds whenever $X$ is a free metrizable $G$-space (see \cite[Proposition 3.5, p. 2304]{colman2013equivariant}).

\medskip A $G$-space $X$ is called \textit{$G$-contractible} if $\text{cat}_G(X)=1$, i.e., the identity map $1_X:X\to X$ is $G$-homotopic to a $G$-map $c:X\to X$ such that $c(X)\subset O(x)$, for some $x\in X$. 

\begin{example}\label{exam:1-2}\rm{
\noindent 
\begin{enumerate}
    \item[(1)] If $G$ acts \textit{transitively} on $X$, i.e., $O(x)=X$, for some (and thus for any) $x\in X$, then $\text{cat}_G(X)=1$.

\item[(2)] Given a $G$-space $X$, the \textit{cone} $C(X)$ of $X$, defined as the quotient of $X \times [0,1]$ obtained by collapsing $X \times \{1\}$ to a single point, is a $G$-space, with the action defined by $g[x, t] := [gx, t]$. Moreover, the point $X \times \{1\}$ belongs to the fixed point set $(C(X))^G$. Note that $\operatorname{cat}_G(C(X)) = 1$ for any $G$-space $X$. The $G$-map $H : C(X) \times [0,1] \to C(X)$, given by
\[
H([x,t],s) = [x, (1 - s)t + s],
\]
satisfies $H_0 = \operatorname{id}_{C(X)}$ and $H_1$ is the constant map at the point $X \times \{1\}$.
\item[(3)] Given a $G$-space $X$, the \textit{suspension} $SX$ of $X$, defined as the quotient of $X \times [0,1]$ obtained by collapsing $X \times \{0\}$ to one point and $X \times \{1\}$ to another, is itself a $G$-space, with the action given by $g[x,t] := [gx,t]$. Moreover, the points $X \times \{0\}$ and $X \times \{1\}$ belong to the fixed point set $(SX)^G$. Note that $\operatorname{cat}_G(SX) \leq 2$ for any $G$-space $X$. Indeed, an open $G$-categorical covering is given by
\[
\left\{ q\left(X \times [0, 3/4)\right),\; q\left(X \times (1/4, 1]\right) \right\},
\]
where $q : X \times [0,1] \to SX$ is the quotient map. 
\end{enumerate}
}
\end{example}

The $G$-contractibility condition yields the following statement.

\begin{lemma}\label{lem:g-contractible-connectivity}\rm{
Let $X$ be a $G$-contractible space. 
\begin{enumerate}
\item[(1)] There exists a point $a\in X$ such that for each $x\in X$, there is a continuous path $\alpha:[0,1]\to X^{G_x}$ such that $\alpha(0)=x$, and $\alpha(1)\in Ga$. 
\item[(2)] Given $x, y \in X$, there exist continuous paths $\alpha : [0,1] \to X^{G_x}$ and \linebreak $\beta : [0,1] \to X^{G_y}$ such that $\alpha(0) = x$, $\beta(1) = y$, and $G\alpha(1) = G\beta(0)$.
 \item[(3)] Let $H$ be a closed subgroup of $G$. Given $x,y\in X^H$, there exist continuous paths $\alpha,\beta:[0,1]\to X^H$ such that $\alpha(0)=x$, $\beta(1)=y$, and $G\alpha(1)=G\beta(0)$. 
     \item[(4)] The orbit space $X/G$ is path-connected. 
\end{enumerate}}
\end{lemma}
\begin{proof}
    Let $F:X\times [0,1]\to X$ be a $G$-homotopy such that $F_0=1_X$ and $F_1(X)\subset Ga$, for some $a\in X$. Given $x\in X$, consider the continuous path $\alpha:[0,1]\to X$ defined by \[\alpha(t)=F(x,t).\] Note that $\alpha(0)=x$ and $\alpha(1)\in Ga$. Additionally, since $F$ is a $G$-map, it follows that $\alpha(t)\in X^{G_x}$, for any $t\in [0,1]$. Similarly, given $y\in X$, then is a continuous path $\gamma:[0,1]\to X^{G_y}$ such that $\gamma(0)=y$ and $\gamma(1)\in Ga$. Consider the path $\beta(t)=\gamma(1-t)$. Note that $\beta(t)\in X^{G_y}$, for any $t\in [0,1]$, $\beta(1)=\gamma(0)=y$, and $G\beta(0)=G\gamma(1)=Ga=G\alpha(1)$. 
    
    Furthermore, if $x\in X^H$, then $H\subset G_x$, and thus $X^{G_x}\subset X^H$, which implies that $\alpha(t)\in X^H$, for all $t\in [0,1]$. 
\end{proof}

Let $X$ and $Y$ be $G$-spaces. We say that $X$ \textit{$G$-dominates $Y$} if there exist $G$-maps $\phi:X\to Y$ and $\psi:Y\to X$ such that $\phi\circ\psi\simeq_G 1_Y$. If, in addition, $\psi\circ\phi\simeq_G 1_X$, then $\phi$ and $\psi$ are \textit{$G$-homotopy equivalences,} and $X$ and $Y$ are \textit{$G$-homotopy equivalent}, written $X\simeq_G Y$.

\medskip The following result is well-known and was established in \cite{marzantowicz1989g}.

\begin{proposition}\label{prop-G-domina}\rm{
If $X$ $G$-dominates $Y$, then $\text{cat}_G(X)\geq \text{cat}_G(Y)$. In particular, if $X\simeq_G Y$, then $\text{cat}_G(X)=\text{cat}_G(Y)$.}
\end{proposition}

The following statement is key to the equivariant category theory. 

\begin{lemma}\label{conservation}\rm{(cf. \cite[Lemma 3.14, p. 2305]{colman2013equivariant})
Let $X$ be a $G$-space, and let $x,y\in X$ such that $H:=G_x$. There exists a continuous path $\alpha:[0,1]\to X^{H}$ such that $\alpha(0)=x$ and $\alpha(1)\in Gy$ if and only if there exists a $G$-homotopy $F:Gx\times [0,1]\to X$ such that $F_0=incl_{Gx}$ and $F_1(Gx)\subset Gy$.}
\end{lemma}
\begin{proof}
To show the converse, we take $\alpha(t):=F(x,t)$. Note that since $x\in X^H$ and $F$ is a $G$-map, it follows that $\alpha(t)\in X^H$, for any $t\in [0,1]$.    

To show the other implication, we take $F(gx,t):=g\alpha(t)$. Note that since $\alpha(t)\in X^H$, for any $t\in [0,1]$, we have $H\subset G_{\alpha(t)}$, for any $t\in [0,1]$, and hence, $F$ is well-defined. Additionally, $F$ satisfies the desired properties. 
\end{proof}

From \cite[Definition 3.13, p. 2305]{colman2013equivariant}, a $G$-space $X$ is said to be \textit{$G$-connected} if the $H$-fixed point set $X^H$ is path-connected, for every closed subgroup $H$ of $G$. 

\medskip Note that if $X$ is $G$-connected, and if $x,y\in X$ are such that $H:=G_x\subset G_y$ (equivalently, $y\in X^H$), then there exists a continuous path $\alpha:[0,1]\to X^{H}$ such that $\alpha(0)=x$ and $\alpha(1)=y\in Gy$. Motivated by Lemma~\ref{conservation}, we present the following definition.

\begin{definition}[$y$-connected]\label{defn:y-connected}
Let $X$ be a $G$-space, and let $y\in X$. We say that $X$ is \textit{$y$-connected} if for each point $x\in X$, there is a continuous path $\alpha:[0,1]\to X^H$, where $H=G_x$, such that $\alpha(0)=x$ and $\alpha(1)\in Gy$.    
\end{definition}

Note that if $X$ is $y$-connected, then  $X$ is also $y'$-connected for any $y'\in Gy$. Moreover, if $X$ is $G$-connected, then $X$ is $x_0$-connected for any $x_0\in X^G$. Furthermore, if $X$ is $x_0$-connected for some point $x_0\in X^G$, then $X$ is $G$-connected.

\medskip It is important to note that there exist $G$-contractible spaces which are not $G$-connected (see Example~\ref{exam:1-2}(1)). However, we have the following remark. 

\begin{remark}\label{rem:g-contractil-connected}
    By Lemma~\ref{lem:g-contractible-connectivity}(1), if $X$ is $G$-contractible, then $X$ is $y$-connected for some $y\in X$. Moreover, if $X$ is $G$-contractible to a fixed point $x_0\in X^G$, then $X$ is $G$-connected.
\end{remark}

By Remark~\ref{rem:g-contractil-connected}, we can conclude that if a $G$-space $X$ is not $y$-connected for any $y\in X$, then $\text{cat}_G(X)\geq 2$. 

\medskip We have the following example.  

\begin{example}[Standard Hamiltonian action]\label{acao-hamil}\rm{
Let $G=S^1$ act on $X=S^2$ by rotation about the $z$-axis. This is known as the \textit{standard Hamiltonian action} of $S^1$ on $S^2$ \cite[Example 8.39, p. 248]{cornea2003lusternik}. The fixed point set of $X$ is $X^G=\{p_N,p_S\}\neq\emptyset$, where $p_N$ is the north pole and $p_S$ is the south pole. Hence, $X$ is not $G$-connected. Furthermore, $X$ is not $y$-connected for any $y\in X$ (and, of course, it is not $G$-contractible). Hence, we conclude that  $\text{cat}_{S^1}(S^2)=2$. Its equivariant categorical covering is given by $U_1=S^2-\{p_N\}$ and $U_2=S^2-\{p_S\}$.}
\end{example}

Let $X$ be a $G$-space and $A$ be a $G$-invariant set of $X$. From \cite[Definition 2.2, p. 103]{lubawski2015}, the \textit{$A$-Lusternik-Schnirelmann $G$-category} of $X$, denoted by $_A\text{cat}_G(X)$, is the least integer $k$ such that $X$ can be covered by open $G$-invariant sets $U_1,\ldots,U_k$ of $X$ such that each inclusion map $\text{incl}_{U_i}:U_i\hookrightarrow X$ is $G$-homotopic to a $G$-map $c_i:U_i\to X$ with $c_i(U_i)\subset A$. If no such integer $k$ exists, we set $_A\text{cat}_G(X)=\infty$. Note that $_{O(x)}\text{cat}_G(X)\geq \text{cat}_G(X)$, for any $x\in X$. 

\medskip Further discussion of $_A\text{cat}_G(X)$ can be found in  \cite{lubawski2015} and \cite{bayeh2019}. 

\medskip We have the following statement, which shows that under $y$-connectivity, the inclusion map of any $G$-categorical set is $G$-homotopic to a $G$-map with values in the orbit of $y$. As a direct consequence, the $A$-Lusternik-Schnirelmann $G$-category satisfies the following property: there exists a subset $A\subseteq X$
such that $_Acat_G(X)=cat_G(X)$, whenever $X$ is $y$-connected. 

\begin{proposition}\label{o-cat-leq-equicat}\rm{
Let $X$ be a $G$-space such that $X$ is $y$-connected. 
\begin{enumerate}
    \item[(1)] If $U\subset X$ is $G$-categorical, then the inclusion map $\text{incl}_{U}:U\hookrightarrow X$ is $G$-homotopic to a $G$-map $s:U\to X$ such that $s(U)\subset Gy$.
    \item[(2)] The following equality holds: \[_{O(y)}\text{cat}_G(X)=\text{cat}_G(X).\] 
\end{enumerate}}
\end{proposition}
\begin{proof}
    \noindent\begin{enumerate}
        \item[(1)] Denote by $H:U\times [0,1]\to X$ the $G$-homotopy between the inclusion $\text{incl}_U:U\hookrightarrow X$ and a $G$-map $c:U\to X$ such that $c(U)\subset Gx$, for some $x\in X$. Since $X$ is $y$-connected, by Lemma \ref{conservation}, we have a $G$-homotopy $F:Gx\times [0,1]\to X$ such that $F_0=\text{incl}_{Gx}$ and $F_1(Gx)\subset Gy$.

Define a homotopy $T:U\times [0,1]\to X$ given by:\[T(x,t):=\left\{
  \begin{array}{ll}
    H(x,2t), & \hbox{ if $0\leq t\leq 1/2$;} \\
   F(c(x),2t-1)), & \hbox{ if $1/2\leq t\leq 1$.}
  \end{array}
\right.\] We have that $T$ is well-defined, it is $G$-equivariant, and it is a homotopy of the inclusion $\text{incl}_{U}:U\hookrightarrow X$ into a map $s:=T_1:U\to X$ such that $s(U)\subset Gy$.
\item[(2)] By Item (1), we obtain $_{O(x)}\text{cat}_G(X)\leq \text{cat}_G(X)$. In addition, recall that the other inequality always holds.
    \end{enumerate}
\end{proof}

Note that Proposition~\ref{o-cat-leq-equicat} generalizes \cite[Remark 2.3, p. 103]{lubawski2015}, where the authors considered more restrictive hypotheses. 

\medskip Additionally, if $X$ is $x_0$-connected for some point $x_0\in X^G$ (equivalently, if $X$ is $G$-connected and $X^G\neq\varnothing$), then $_{\{x_0\}}\text{cat}_G(X)\geq \text{cat}(X)$. Hence, Proposition~\ref{o-cat-leq-equicat} implies that the nonequivariant category provides a lower bound for the equivariant category, i.e., $\text{cat}(X)\leq \text{cat}_G(X)$ for any $G$-connected space $X$ with $X^G\neq\varnothing$.

\medskip The following definition is the equivariant version of cofibration and retract in the category of $G$-spaces.

\begin{definition}[$G$-cofibration]\label{defn:g-cofibration}
Let $X$ be a $G$-space and $A\subset X$ be an invariant subset. 
\begin{enumerate}
    \item[(1)] The inclusion map $j:A\hookrightarrow X$ is a \textit{$G$-cofibration} if for every $G$-space $Y$, every $G$-map $f:X\to Y$, and every $G$-homotopy $H:A\times [0,1]\to Y$ satisfying $H(a,0)=f\circ j(a)$, for all $a\in A$, there exists a $G$-homotopy $\widehat{H}:X\times [0,1]\to Y$ such that $\widehat{H}(j(a),t)=H(a,t)$, for all $a\in A$ and $t\in [0,1]$, and $\widehat{H}(x,0)=f(x)$, for all $x\in X$. We recall that $G$ acts trivially on $I=[0,1]$ and diagonally on $X\times [0,1]$.
\begin{eqnarray*}
\xymatrix{ A \ar@{^{(}->}[rr]^{\,\,j} \ar[dd]_{\iota_0}  &   &  X\ar[dd]^{f}\ar[dl]_{\iota_0}\\
& X\times [0,1] \ar@{-->}[dr]^{\,\,\widehat{H}}  &  \\
 A\times [0,1]\ar[ur]^{\,\,j\times 1_I}\ar[rr]_{H} &    & Y}
\end{eqnarray*} where $\iota_0(x)=(x,0)$ for all $x$.
\item[(2)]  $A$ is a \textit{$G$-retract} of $X$ if there exists a $G$-map $r:X\to A$ such that $r(a)=a$, for all $a\in A$. Such a map $r$ is called a \textit{$G$-retraction}. 
\end{enumerate} 
\end{definition}

We have a characterization of $G$-cofibrations in terms of $G$-retracts, which leads to a key property of $G$-cofibrations. The proof of this claim follows analogously to the nonequivariant case, as presented in \cite[Theorem 4.1.16, Theorem 4.1.7]{aguilar2002algebraic}.

\begin{proposition}\rm{
Let $X$ be a $G$-space and let $A\subset X$ be a closed invariant subset. Then: 
\begin{enumerate}
    \item[(1)]\label{caracterizacao-Gcofibration} The inclusion map $j:A\hookrightarrow X$ is a $G$-cofibration if and only if  $X\times \{0\}\cup A\times I$ is a $G$-retract of $X\times I$. 
$$
\begin{tikzpicture}
\filldraw[fill=black!20!white, draw=black] (-1,0) rectangle (1,1);
\draw(-3,0)--(3,0); 
\draw[dashed](-3,0)--(-3,1); \draw[dashed](-3,1)--(-1,1); \draw[dashed](1,1)--(3,1); \draw[dashed](3,1)--(3,0);
\node [below] at (3,0) {\tiny$X$}; \node [below] at (0,0) {\tiny$A$};
\node at (0,0.5) {\tiny$A\times [0,1]$};
\end{tikzpicture}
$$
\item[(2)]\label{G-cofibration-implica-homotopia} If the inclusion map $j:A\hookrightarrow X$ is a $G$-cofibration, then there exists a  $G$-homotopy $D:X\times [0,1]\to X$ and a $G$-map  $\varphi:X\to [0,1]$ such that $A\subset \varphi^{-1}(0)$ and: \begin{itemize}
\item $D(x,0)=x$, for all $x\in X$.
\item $D(a,t)=a$, for all $a\in A$ and for all $t\in [0,1]$.
\item $D(x,t)\in A$, for all $x\in X$ and for all $t>\varphi(x)$. 
\end{itemize}
\end{enumerate}}
\end{proposition}
\begin{proof}
\noindent\begin{enumerate}
    \item[(1)] If $j:A\hookrightarrow X$ is a $G$-cofibration, then the $G$-map $f:X\to X\times \{0\}\cup A\times I$ given by $f(x)=(x,0)$, for all $x\in X$, and the $G$-homotopy $H:A\times I\to X\times \{0\}\cup A\times I$ given by $H(a,t)=(a,t)$, for all $(a,t)\in A\times I$, together define a $G$-map $\widehat{H}:X\times I\to X\times \{0\}\cup A\times I$, which serves as a $G$-retraction. This means we can take $r=\widehat{H}$.

Conversely, if we have a $G$-retraction $r:X\times I\to X\times \{0\}\cup A\times I$, then for any $G$-space $Y$, any $G$-map $f:X\to Y$, and any $G$-homotopy $H:A\times I\to Y$ satisfying $H(a,0)=f\circ j(a)$, for all $a\in A$, we can define a $G$-homotopy $\widehat{H}:X\times I\to Y$ by:
\[\widehat{H}(x,t):=\left\{
  \begin{array}{ll}
    f\circ proj_X\circ r(x,t), & \hbox{if $(x,t)\in r^{-1}(X\times\{0\})$;} \\
   H\circ r(x,t), & \hbox{if $(x,t)\in r^{-1}(A\times I)$.}
  \end{array}
\right.\] We note that $\widehat{H}$ is continuous, since $X\times\{0\}$ and $A\times I$ are closed in $X\times I$.
\item[(2)] We shall use the characterization given in Item (1), namely, that there exists a $G$-retraction $r:X\times I\to X\times \{0\}\cup A\times I$. We define $\varphi$ and $D$ as follows:
\begin{eqnarray*}
\varphi(x) &:=& \sup_{t\in I}\mid t-proj_I r(x,t)\mid,~ \text{ for } x\in X;\\
D(x,t) &:=& proj_X r(x,t), ~\text{ for } x\in X, ~t\in I.
\end{eqnarray*} We have that $\varphi$ is well-defined and is a $G$-map, since for any $g\in G$ and any $x\in X$, we have:
\begin{eqnarray*}
\varphi(gx) &=&  \sup_{t\in I}\mid t-proj_I r(gx,t)\mid\\
&=& \sup_{t\in I}\mid t-proj_I (gr(x,t))\mid\\
&=& \sup_{t\in I}\mid t-proj_I r(x,t)\mid\\
&=& \varphi(x)\\
&=& g\varphi(x).
\end{eqnarray*} Here, we recall that $G$ acts trivially on $I$ and diagonally on $X\times I$. Furthermore, $D$ is well-defined, equivariant, and satisfies the desired properties. Note that for $t>\varphi(x)$, we have $r(x,t)\in A\times I$, since $t>0$. Thus, $D(x,t)\in A$. 
\end{enumerate} 
\end{proof}

Now, we introduce the notion of a $G$-well-pointed space.  

\begin{definition}\label{defn:g-well-pointed}
A $G$-space $X$ has a \textit{$G$-non-degenerate basepoint} $x_0$ if $x_0\in X^G$ and the inclusion $\{x_0\}\hookrightarrow X$ is a $G$-cofibration. A $G$-space is \textit{$G$-well-pointed} if it has a $G$-non-degenerate basepoint.
\end{definition}

Examples of equivariant cofibrations are usually a CW pair $(X,A)$, where some group $G$ acts on $X$ (cellularly) \cite[Theorem 3.1]{may1996}. Thus, every $G$-CW-complex with a $G$-fixed basepoint is $G$-wellpointed, since the inclusion of the basepoint is a $G$-cofibration.

\medskip Proposition~\ref{G-cofibration-implica-homotopia} implies that any $G$-non-degenerate basepoint admits a \aspas{good} open $G$-categorical neighborhood.   

\begin{proposition}\label{G-ponto-base-nao-degenerado}\rm{
If $X$ is a $G$-space with a $G$-non-degenerate basepoint $x_0$, then there exists an open invariant neighborhood $N$ of $x_0$ that \textit{$G$-contracts} to $x_0$ in $X$ \textit{relative} to $x_0$, that is, there exists a $G$-homotopy $H:N\times [0,1]\to X$ such that $H_0=\text{incl}_N$, $H_1(N)\subset O(x_0)=\{x_0\}$, and $H(x_0,t)=x_0$, for all $t\in [0,1]$.}
\end{proposition}
\begin{proof}
We can use Proposition \ref{G-cofibration-implica-homotopia} for the closed invariant $A=\{x_0\}$. Then, we define $N:=\varphi^{-1}([0,1))\subset X$ and let \[H:=D\mid_{N\times [0,1]}:N\times [0,1]\to X\] be the restriction of $D$. Thus, $N$ is an open invariant neighborhood $N$ of $x_0$ that $G$-contracts to $x_0$ in $X$ relative to $x_0$.
\end{proof}

Now, just as in the nonequivariant case, we also need some separation conditions. Recall that a Hausdorff space $X$ is called \textit{normal} if whenever $A,B\subset X$ are closed sets such that $A\cap B=\emptyset$, there exist disjoint open sets $U,V\subset X$ such that $A\subset U$ and $B\subset V$.

\medskip The following definition is the equivariant version of normal spaces.

\begin{definition}[$G$-normal space]\label{defn:g-normal}
A $G$-space $X$ is called \textit{$G$-normal}  if whenever $A,B\subset X$ are closed invariant sets such that $A\cap B=\emptyset$, there exist  disjoint open invariant sets $U,V\subset X$ such that $A\subset U$ and $B\subset V$.
\end{definition}

The following statement is the normal version of \cite[Lemma 3.12, p. 2305]{colman2013equivariant}.

\begin{proposition}\label{prop:normal-g-g-normal} \rm{
If $X$ is a normal $G$-space, then $X$ is $G$-normal.}
\end{proposition}
\begin{proof}
Recall that the quotient map $p:X\to X/G$ is closed. If $X$ is normal, then $X/G$ is normal (see \cite[Theorem 3.3, p. 145]{dugundji1966general}). Thus, it suffices to show that the normality of $X/G$ implies the $G$-normality of $X$. Let $A,B\subset X$ be disjoint closed invariant sets. Using the quotient map $p:X\to X/G$, we see that $p(A), p(B)\subset X/G$ are disjoint closed sets because $A$ and $B$ are disjoint invariant sets. Since $X/G$ is normal, there exist disjoint open sets $U,V\subset X/G$ such that $p(A)\subset U$ and $p(B)\subset V$. 

Next, we observe that: \[A\subset (p^{-1}\circ p)(A)\subset p^{-1}(U), \quad\quad B\subset (p^{-1}\circ p)(B)\subset p^{-1}(V).\] Since $p^{-1}(U)$ and $p^{-1}(V)$ are invariant, they are open invariant sets in $X$. Therefore, there exist disjoint open invariant sets $p^{-1}(U)$ and $p^{-1}(V)\subset X$ such that $A\subset p^{-1}(U)$ and $B\subset p^{-1}(V)$. This proves that $X$ is $G$-normal.
\end{proof}

Note that any $G$-normal space satisfies the following properties, which are an extension of the nonequivariant case (cf. \cite[Theorem 6.1, pg. 152]{dugundji1966general}). 

\begin{proposition}\rm{
Let $X$ be a $G$-space.
\begin{enumerate}
    \item[(1)]\label{primera-prop-G-normal} If $X$ is $G$-normal, then for each invariant closed subset $A\subset X$ and nonempty invariant open set $U\supset A$, there exists an invariant open set $V$ such that $A\subset V\subset \overline{V}\subset U$.
    \item[(2)]\label{covering-characterization-G-normality} $X$ is $G$-normal if and only if for any covering $\{U_i\}_{i=1}^n$ of $X$ by open invariant sets, there exists a covering $\{V_i\}_{i=1}^n$ of $X$ by open invariant sets such that $\overline{V_i}\subset U_i$, for all $i$.
\end{enumerate} }
\end{proposition}
\begin{proof}
\noindent\begin{enumerate}
    \item[(1)] Recall that $X-U$ is invariant because $U$ is invariant. Since $X$ is $G$-normal, and $A$ and $B:=X-U$ are disjoint closed invariant subsets of $X$ (with $U$ being nonempty), there exist disjoint open invariant sets $F$ and $E$ such that $A\subset F$ and $B\subset E$. Define $V:=F$. Then, we have $\overline{V}\subset U$, as required. 
    \item[(2)] $(\Longrightarrow)$: For each $x\in X$, define $h(x):=\max\{i:~x\in U_i\}$. We will define the sets $\{V_i\}_{i=1}^n$ by induction on $n$.

Define $F_1:=X-(U_2\cup\cdots\cup U_n)\subset U_1$. First, note that $F_1$ is closed and invariant. By item (1), there exists an open invariant set $V_1$ such that $F_1\subset V_1\subset\overline{V_1}\subset U_1$. Note that $V_1\cup U_2\cup\cdots\cup U_n=X$. Indeed, given $x\in X$, if $h(x)>1$, then $x\in U_2\cup\cdots\cup U_n$, and if $h(x)\leq 1$, then $x\notin U_2\cup\cdots\cup U_n$, i.e., $x\in F_1\subset V_1$. Therefore, $V_1\cup U_2\cup\cdots\cup U_n=X$.

Now, define $F_2:=X-(V_1\cup U_3\cup\cdots\cup U_n)\subset U_2$. Note that $F_2$ is closed and invariant. By item (1), there is an open invariant $V_2$ with $F_2\subset V_2\subset\overline{V_2}\subset U_2$. Note that $V_1\cup V_2\cup U_3\cup\cdots\cup U_n=X$. Indeed, given $x\in X$, if $h(x)>2$, then $x\in U_3\cup\cdots\cup U_n$, and if $h(x)\leq 2$, then $x\notin U_3\cup\cdots\cup U_n$. Therefore, $V_1\cup V_2\cup U_3\cup\cdots\cup U_n=X$.

Assume that $V_1, V_2,\ldots, V_k$ are open invariant sets such that \[V_1\cup\cdots\cup V_k\cup U_{k+1}\cup\cdots\cup U_n=X.\] Define $F_{k+1}:=X-(V_1\cup\cdots\cup V_k\cup U_{k+2}\cup\cdots\cup U_n)\subset U_{k+1}$. Note that $F_{k+1}$ is closed and invariant. By item (1), there exists an open invariant set $V_{k+1}$ such that $F_{k+1}\subset V_{k+1}\subset\overline{V_{k+1}}\subset U_{k+1}$. Note that $V_1\cup\cdots \cup V_k\cup V_{k+1}\cup U_{k+2}\cup\cdots\cup U_n=X$. Indeed, given $x\in X$, if $h(x)>k+1$, then $x\in U_{k+2}\cup\cdots\cup U_n$, and if $h(x)\leq k+1$, then $x\notin U_{k+2}\cup\cdots\cup U_n$. Therefore, \[V_1\cup\cdots \cup V_k\cup V_{k+1}\cup U_{k+2}\cup\cdots\cup U_n=X.\]

Finally, assume that $V_1, V_2,\ldots, V_{n-1}$ are open invariant sets such that \[V_1\cup\cdots\cup V_{n-1}\cup U_n=X.\] Define $F_{n}:=X-(V_1\cup\cdots\cup V_{n-1})\subset U_{n}$. Note that $F_{n}$ is closed and invariant. By item (1), there exists an open invariant $V_{n}$ such that $F_{n}\subset V_{n}\subset\overline{V_{n}}\subset U_{n}$. Note that $V_1\cup\cdots \cup V_n=X$. Indeed, given $x\in X$, if $x\in V_1\cup\cdots\cup V_{n-1}$ the equality holds, and if $x\notin V_1\cup\cdots\cup V_{n-1}$, then $x\in F_{n}\subset V_{n}$. Therefore, $V_1\cup\cdots \cup V_n=X$.

Thus, there exists a covering $\{V_i\}_{i=1}^n$ of $X$ by open invariant sets such that $\overline{V_i}\subset U_i$, for all $i$.

$(\Longleftarrow)$: Let $A,B\subset X$ be disjoint closed invariant sets in $X$. Then \[\{X-A, X-B\}\] is a covering of $X$ by open invariant sets. Hence, there exist open invariant sets $V_1$ and $V_2$ in $X$ such that $\overline{V_1}\subset X-A$, $\overline{V_2}\subset X-B$, and $V_1\cup V_2=X$. Note that $U:=X-\overline{V_1}$ and $V:=X-\overline{V_2}$ are open invariant neighborhoods of $A$ and $B$, respectively, and $U\cap V=X-(\overline{V_1}\cup \overline{V_2})=\emptyset$. Thus, $X$ is $G$-normal. 
    \end{enumerate}
\end{proof}


\section{Equivariant Category of Wedges}\label{sec:equiv-wedges}

In this section, we use the previous results to derive important claims that will imply our first main theorem (Theorem \ref{principal}).

\medskip First, we show that for any $G$-connected, $G$-normal space with a $G$-non-degenerate basepoint, there exists a \aspas{good} open $G$-categorical cover. This claim is the equivariant version of \cite[Lemma 1.25, pg. 13]{cornea2003lusternik}.

\begin{proposition}\label{G-categorico-baseado}\rm{
Suppose $X$ is a $G$-connected, $G$-normal space with a $G$-non-degenerate basepoint $x_0$. If $\text{cat}_G(X)\leq n$, then there is an open $G$-categorical cover $V_1,\ldots,V_n$ such that $x_0\in V_j$, for all $j=1,\ldots,n$, and each $V_j$ is $G$-contractible to $O(x_0)=\{x_0\}$ relative to $x_0$, that is, there is a $G$-homotopy $T_j:V_j\times [0,1]\to X$ such that $T_j(x,0)=x$, for all $x\in V_j$, $T_j(x,1)=x_0$, for all $x\in V_j$, and $T_j(x_0,t)=x_0$, for all $t\in [0,1]$.}
\end{proposition}
\begin{proof}
Let $\{U_1,\ldots,U_n\}$ be an open $G$-categorical cover of $X$. By Proposition~\ref{o-cat-leq-equicat}(1), we may assume that the inclusions $\text{incl}_{U_j}:U_j\hookrightarrow X$ are all $G$-homotopic into $O(x_0)=\{x_0\}$, with respective $G$-contracting $H_j:U_j\times [0,1]\to X$. Note that by the $G$-normality of $X$, there exists a refined cover $\{W_j\}_{j=1}^{n}$ of $\{U_j\}_{j=1}^{n}$ by open invariant sets with $W_j\subset \overline{W_j}\subset U_j$,  for each $j=1,\ldots,n$ (see Proposition~\ref{covering-characterization-G-normality}(2)).

By the $G$-non-degenerate basepoint hypothesis (see Proposition \ref{G-ponto-base-nao-degenerado}), there exists an open invariant neighborhood of $x_0$, denoted $N$, and a $G$-contracting homotopy $H:N\times [0,1]\to X$ with $H_0=\text{incl}_N$, $H_1(N)\subset O(x_0)=\{x_0\}$, and $H(x_0,t)=x_0$, for all $t\in [0,1]$.

Without loss of generality, we can assume that $x_0\in U_j$, for $j=1,\ldots,k$, for some $1\leq k\leq n-1$, and $x_0\notin U_j$, for $j=k+1,\ldots,n$. Define a new $G$-contracting open invariant neighborhood of $x_0$ by \[
\mathcal{N} = N\cap U_1\cap\cdots\cap U_k\cap(X-\overline{W}_{k+1})\cap\cdots\cap(X-\overline{W}_{n})\subset N.
\] Note that $x_0\in \mathcal{N}$ (since $x_0\in X-U_j\subset X-\overline{W}_{j}$ for $j=k+1,\ldots,n$) and $\mathcal{N}\cap W_j=\emptyset$ for all $j=k+1,\ldots,n$.

Now, again by the $G$-normality of $X$ (see Proposition~\ref{primera-prop-G-normal}(1)), for the closed invariant set $\{x_0\}\subset\mathcal{N}$ and the open invariant set $\mathcal{N}$, there exists an open invariant set $M$ of $X$ with $x_0\in M\subset \overline{M}\subset \mathcal{N}$. Note that $\mathcal{N}\subset U_j$, for each $j=1,\ldots,k$.

Now, we can define an open $G$-categorical cover of $X$ that satisfies the conclusions of the proposition. Let \[
V_j :=\begin{cases}
    (U_j\cap (X-\overline{M}))\cup M, & \hbox{for $j=1,\ldots,k$;} \\
   W_j\cup \mathcal{N}, & \hbox{for $j=k+1,\ldots,n$.}
  \end{cases}
 \] Note that $\{V_j\}_{j=1}^{n}$ covers $X$. This follows because \[[(U_j\cap (X-\overline{M}))\cup M]\cup (\overline{M}-M)=U_j\] and $\overline{M}\subset \mathcal{N}$ implies that \[[(U_j\cap (X-\overline{M}))\cup M]\cup \mathcal{N}\supset U_j.\] Moreover, since $W_j\subset U_j$, we have that  \[[(U_j\cap (X-\overline{M}))\cup M]\cup \mathcal{N}\supset W_j,\] for each $j=1,\ldots,k$. Thus, $\bigcup_{j=1}^nV_j\supset \bigcup_{j=1}^n W_j$, and since $\{W_j\}_{j=1}^{n}$ covers $X$, we conclude that $\{V_j\}_{j=1}^{n}$ covers $X$. Furthermore, $\{V_j\}_{j=1}^{n}$ is a covering by open invariant sets. Note that $x_0\in V_j$, for all $j=1,\ldots,n$, since $x_0\in M$ and $x_0\in \mathcal{N}$. 
 
 Moreover, each $V_j$, for $j=1,\ldots,n$ consists of two disjoint open invariant subsets: one subset of $U_j$ not containing the basepoint $x_0$, and one subset of $\mathcal{N}$ containing the basepoint. This allows us to define a $G$-contracting homotopy: for $j=1,\ldots,k$, define $T_j:V_j\times [0,1]\to X$ by
 \[ T_j(x,t) :=\left\{
  \begin{array}{ll}
   H_j(x,t), & \hbox{ if $x\in U_j\cap (X-\overline{M})$;} \\
   H(x,t), & \hbox{ if $x\in M$.}
  \end{array}
\right.\] We observe that $T_j$ is well-defined, invariant, and it is a homotopy between the inclusion $\text{incl}_{V_j}$ and the constant map $\overline{x_0}:[0,1]\to X$ with $\overline{x_0}(t)=x_0,\forall t\in [0,1]$. Furthermore, for any $t\in [0,1]$, we have $T_i(x_0,t)=H(x_0,t)$, since $x_0\in M$. 

For $j=k+1,\ldots,n$, we define  $T_j:V_j\times [0,1]\to X$ by
 \[ T_j(x,t) :=\left\{
  \begin{array}{ll}
   H_j(x,t), & \hbox{ if $x\in W_j$;} \\
   H(x,t), & \hbox{ if $x\in \mathcal{N}$.}
  \end{array}
\right.\] Note that $T_j$ is well-defined, invariant, and it is a homotopy between the inclusion $\text{incl}_{V_j}$ and the constant map $\overline{x_0}:[0,1]\to X,~\overline{x_0}(t)=x_0,\forall t\in [0,1]$. Furthermore, for any $t\in [0,1]$, we have $T_j(x_0,t)=H(x_0,t)$, since $x_0\in\mathcal{N}$. 

\end{proof}

Thus, we introduce the following definition.

\begin{definition}[Based $G$-categorical covering]\label{defn-based-g-cat}
An open $G$-categorical covering $V_1,\ldots,V_n$ of $X$ such that $x_0\in V_i$, for all $i=1,\ldots,n$, and each $V_i$ is $G$-contractible to $O(x_0)=\{x_0\}$ relative to $x_0$, is called a \textit{based $G$-categorical covering}.
\end{definition}

For $G$-spaces $X$ and $Y$, we will assume that $G$ acts diagonally on $X\times Y$, meaning that for all $(x,y)\in X\times Y$ and for all $g\in G$, we have  $g(x,y):=(gx,gy)$. If $x_0\in X^G$ and $y_0\in Y^G$, the wedge sum $X\vee Y=X\times\{y_0\}\cup\{x_0\}\times Y\subset X\times X$ is an invariant subset of the $G$-space $X\times Y$. Furthermore, note that for any $x\in X$ and $y\in Y$, the orbits in $X\vee Y$ are $O(x,y_0)=O(x)\times\{y_0\}$ and $O(x_0,y)=\{x_0\}\times O(y)$. In addition, observe that the wedge $X\vee Y$ coincides with the space obtained from the disjoint union $X\sqcup Y$ by identifying the basepoints $x_0$ and $y_0$.

\begin{proposition}\label{prop:lower-bound-equiv-wedge}\rm{
  Let $X$ and $Y$ be $G$-spaces with $x_0\in X^G$ and $y_0\in Y^G$. Then, the wedge $X\vee Y$ satisfies the following inequality:
  \[\max\{\text{cat}_G(X),\text{cat}_G(Y)\}\leq \text{cat}_G(X\vee Y).\]}
\end{proposition}
\begin{proof}
    Note that there are $G$-retractions $r_X:X\vee Y\to X$ and $r_Y:X\vee Y\to Y$. Indeed, the map $r_X:X\vee Y\to X$ is given by  
\[
r_X(x)=\left\{
  \begin{array}{ll}
    x, & \hbox{if $x\in X$;} \\
    x_0, & \hbox{if $x\in Y$.} 
    \end{array}
\right.
\] Note that $r_X$ is well-defined, is equivariant, and $r_X(z)=z$, for all $z\in X$. Hence, $r_X$ is a $G$-retraction. Similarly, the map $r_Y:X\vee Y\to Y$, given by \[
r_Y(y)=\left\{
  \begin{array}{ll}
    y_0, & \hbox{if $y\in X$;} \\
    y, & \hbox{if $y\in Y$.} 
    \end{array}
\right.
\] is a $G$-retraction. By Proposition \ref{prop-G-domina}, we then have $\text{cat}_G(X)\leq \text{cat}_G(X\vee Y)$ and $\text{cat}_G(Y)\leq \text{cat}_G(X\vee Y)$. Thus, \[\max\{\text{cat}_G(X),\text{cat}_G(Y)\}\leq \text{cat}_G(X\vee Y).\]
\end{proof}

Then, we state one of the main purposes of this paper. This claim is the equivariant version of \cite[Proposition 1.27, pg. 14]{cornea2003lusternik}. In the next section, we will use this theorem to compute the equivariant and invariant topological complexities for $G$-connected $G$-CW-complexes with $X^G\neq\varnothing$. 

\begin{theorem}\label{principal}\rm{
If $X,Y$ are $G$-connected, $G$-normal spaces with $G$-non-degenerate basepoints, then
\[
\text{cat}_G(X\vee Y)=\max\{\text{cat}_G(X),\text{cat}_G(Y)\},
\] where $X\vee Y$ is the wedge of the disjoint sets $X$ and $Y$, obtained by identifying their basepoints.}
\end{theorem}

\begin{proof}
Let $\text{cat}_G(X)=n$ and $\text{cat}_G(Y)=m$, with based $G$-categorical covers $\{U_i\}_{i=1}^n$ and $\{V_j\}_{j=1}^m$, respectively (see Proposition~\ref{G-categorico-baseado}). Without loss of generality, assume  $n\leq m$. Define
\[W_i:=\left\{
  \begin{array}{ll}
   U_i\cup V_i, & \hbox{ for $i=1,\ldots,n$;} \\
   U_n\cup V_i, & \hbox{ for $i=n+1,\ldots,m$.}
  \end{array}
\right.\] Note that each $W_i$ is open in $X\vee Y$, since $U_i$ is open in $X$ and $V_j$ is open in $Y$, and $U_i\cap V_j=\{x_0=y_0\}$. Each $W_i$ is invariant, as both $U_i$ and $V_j$ are invariant. 

Let $H_i$ and $F_j$ be $G$-contractible homotopies for $U_i$ and $V_j$, respectively. For each $i=1,\ldots,n$, define $T_i:W_i\times [0,1]\to X\vee Y$ by
\[T_i(x,t):=\left\{
  \begin{array}{ll}
   H_i(x,t), & \hbox{ if $x\in U_i$;} \\
   F_i(x,t), & \hbox{ if $x\in V_i$.}
  \end{array}
\right.\]

For each $i=n+1,\ldots,m$, define $T_i:W_i\times [0,1]\to X\vee Y$ by
\[T_i(x,t):=\left\{
  \begin{array}{ll}
   H_n(x,t), & \hbox{ if $x\in U_n$;} \\
   F_i(x,t), & \hbox{ if $x\in V_i$.}
  \end{array}
\right.\] Then, each $T_i$ is well-defined, invariant, and it is a homotopy such that $(T_i)_0=\text{incl}_{W_i}$ and $(T_i)_1(W_i)\subset O(x_0=y_0)=\{x_0=y_0\}$. Hence, $\{W_i\}_{i=1}^m$ is an open $G$-categorical cover of $X\vee Y$. Therefore, $\text{cat}_G(X\vee Y)\leq m=\max\{n,m\}=\max\{\text{cat}_G(X),\text{cat}_G(Y)\}$.

For the reverse inequality, see Proposition~\ref{prop:lower-bound-equiv-wedge}. Consequently, \[\text{cat}_G(X\vee Y)=\max\{\text{cat}_G(X),\text{cat}_G(Y)\}.\]
\end{proof}

A direct consequence of Theorem~\ref{principal} is the following example.

\begin{example}\label{equiv-cat-x-x}\rm{
For any $G$-connected, $G$-normal space $X$ with a $G$-non-degenerate basepoint, by Theorem~\ref{principal}, we have \[\text{cat}_G(X\vee X)=\text{cat}_G(X).\]}
\end{example}

As a first application of Theorem~\ref{principal}, we characterize the $G$-contractibility of the wedge $\bigvee_{i=1}^m X_i$.

\begin{corollary}\label{cor-wedge}
Let $\{X_i\}_{i=1}^m$ be a collection of $G$-connected, $G$-normal spaces with $G$-non-degenerate basepoints. The wedge $\bigvee_{i=1}^m X_i$ is $G$-contractible if and only if each $X_i$ is $G$-contractible, for $i=1,\ldots, m$.
\end{corollary}
\begin{proof}
By Theorem~\ref{principal}, we have $\text{cat}_G(\bigvee_{i=1}^m X_i)=\max\{\text{cat}_G(X_i):~i=1,\ldots,m\}$. Therefore, $\text{cat}_G(\bigvee_{i=1}^m X_i)=1$ if and only if $\text{cat}_G(X_i)=1$, for each $i=1,\ldots, m$.
\end{proof}

A second application of Theorem~\ref{principal} is to compute the equivariant category of the quotient $X/A$ for a $G$-space $X$ and an invariant subset $A$ such that the inclusion $A\hookrightarrow X$ is $G$-homotopic to a constant map $\overline{x_0}:A\to X$ for some $x_0\in X^G$. 

\begin{theorem}\label{prop:equiv-quotient}\rm{
Let $X$ be a $G$-space and $A\subset X$ an invariant subset such that the inclusion map $i:A\hookrightarrow X$ is a $G$-cofibration. If the inclusion map $i:A\hookrightarrow X$ is $G$-homotopic to a constant map $\overline{x_0}:A\to X$, for some $x_0\in X^G$, then $X/A$ is $G$-homotopic to the wedge $X\vee S A$, where $A\times\{0\}$ is the basepoint of $S A$ and $x_0$ is the base point of $X$. 

In addition, if $X$ is a non $G$-contractible, $G$-connected, $G$-normal space with a $G$-non-degenerate basepoint, and if $S A$ is a $G$-connected, $G$-normal space with a $G$-non-degenerate basepoint, then we have \[\text{cat}_G(X/A)=\text{cat}_G(X).\]}
\end{theorem}
\begin{proof}
The statement that $X/A$ is $G$-homotopy equivalent to $X\vee S A$ follows by analogy to \cite[Example 0.14, p.14]{hatcher2002}. Hence, by Proposition~\ref{prop-G-domina}, we obtain that $\text{cat}_G(X/A)=\text{cat}_G(X\vee S A)$. From Theorem~\ref{principal}, we conclude that $\text{cat}_G(X/A)=\text{cat}_G(X)$, since  $\text{cat}_G(X)\geq 2$ (using the fact that $X$ is not $G$-contractible) and $\text{cat}_G(S A)\leq 2$ (see Example~\ref{exam:1-2}(3)).
\end{proof}

We can consider more general product actions. 

\begin{remark}\label{rem:product-action}
Let $K$ be another compact Hausdorff group. Then, the product of a $G$-space $X$ and a $K$-space $Y$ becomes a $G\times K$-space in an obvious way, namely, $(g,k)(x,y):=(gx,ky)$, for all $x\in X$, $y\in Y$, $g\in G$, and $k\in K$. If $x_0\in X^G$ and $y_0\in Y^K$, the wedge $X\vee Y=X\times\{y_0\}\cup\{x_0\}\times Y\subset X\times X$ is an invariant subset of the $G\times K$-space $X\times Y$. Indeed, for any $x\in X$, $y\in Y$, and any $g\in G$, $k\in K$, we have $(g,k)(x,y_0)=(gx,ky_0)=(gx,y_0)\in X\vee Y$ and $(g,k)(x_0,y)=(gx_0,ky)=(x_0,ky)\in X\vee Y$. Hence, $X\vee Y$ is a $G\times K$-space. Furthermore, note that, for any $x\in X$ and $y\in Y$, the orbits in $X\vee Y$ are $O(x,y_0)=O(x)\times\{y_0\}$ and $O(x_0,y)=\{x_0\}\times O(y)$. Here, $O(x)$ denotes the orbit of $x$ in the $G$-space $X$, and $O(y)$ denotes the orbit of $y$ in the $K$-space $Y$. 
\end{remark}

Hence, in a similar way to Theorem \ref{principal}, one can obtain the following result.

\begin{proposition}\label{prop:equiv-cat-product}\rm{
If $X$ is a $G$-connected, $G$-normal space with a $G$-non-degenerate basepoint, and if $Y$ is a $K$-connected, $K$-normal space with a $K$-non-degenerate basepoint, then
\[
\text{cat}_{G\times K}(X\vee Y)=\max\{\text{cat}_G(X),\text{cat}_K(Y)\},
\] where $X\vee Y$ is the wedge of the disjoint sets $X$ and $Y$, obtained by identifying their basepoints.}
\end{proposition}

Recall that a pointed $G$-space means a $G$-space with the distinguished basepoint fixed by $G$.  

\begin{remark}\label{rem:prod-lower-bound}
Note that the inequality \[\max\{\text{cat}_G(X),\text{cat}_K(Y)\}\leq \text{cat}_{G\times K}(X\vee Y)\] holds for any pointed $G$-space $X$ and pointed $K$-space $Y$. 
\end{remark}

On the other hand, it is natural to inquire about the equivariant category of the smash product. For this, we have the following remark.

\begin{remark}[Smash product]\label{rem:smash-pro}
For the $G$-spaces $X$ and $Y$, we recall that $G$ acts diagonally on $X\times Y$. If $x_0\in X^G$ and $y_0\in Y^G$, the wedge $X\vee Y=X\times\{y_0\}\cup\{x_0\}\times Y\subset X\times X$ is an invariant subset of the $ G$-space $X\times Y$, and the smash product \[X\wedge Y=\dfrac{X\times Y}{X\vee Y}\] is a $G$-space. The value of the equivariant category $\operatorname{cat}_G(X \wedge Y)$ is not known for arbitrary pointed $G$-spaces $X$ and $Y$. However, when $Y = S^k$ (for $k \geq 0$) equipped with the trivial $G$-action, we have the following result.
\end{remark}

\begin{theorem}\label{thm:equiv-smash}\rm{
Let $H_q(-; \mathbb{Z})$ denote the $q$th singular homology group. Let $X$ be a pointed $G$-space such that $H_q(X;\mathbb{Z})\neq 0$, for some $q\geq 1$, and let $S^k$ be the $k$-dimensional sphere ($k\geq 0$) 
 equipped with the trivial $G$-action. Assume that $X \wedge S^k$ is $G$-connected for $k \geq 1$. Then:
\[
\text{cat}_G(X\wedge S^k)=\begin{cases}
\text{cat}_G(X),& \hbox{ for $k=0$,}\\
2,& \hbox{ for $k\geq 1$.}
\end{cases}
\] }
\end{theorem}

\begin{proof}
Recall that $G$ acts trivially on $S^k$. It is a technical procedure to check that there exists a $G$-equivariant homeomorphism between $X\wedge S^0$ and $X$. Hence, by Proposition~\ref{prop-G-domina}, we have $\text{cat}_G(X\wedge S^0)=\text{cat}_G(X)$. 

For the case $k\geq 1$, it is also a technical procedure to check that there exists a $G$-equivariant homeomorphism between $X\wedge S^k$ and the iterated reduced suspension $\Sigma^k X$. Hence, by Proposition~\ref{prop-G-domina}, we have $\text{cat}_G(X\wedge S^k)=\text{cat}_G(\Sigma^k X)$.  Note that $\text{cat}_G(\Sigma^k X)\leq 2$. Then, we have \begin{eqnarray*}
2 &\leq& \text{cat}(\Sigma^k X)\\
&\leq& \text{cat}_G(\Sigma^k X)\\
&\leq& 2.
\end{eqnarray*} The first inequality follows from the fact that $H_q(X;\mathbb{Z})\neq 0$, for some $q\geq 1$ (using the Mayer–Vietoris sequence, cf. \cite[Lemma 4.10, p 36]{zapata2018}).
\end{proof}

\section{Equivariant and Invariant Topological Complexities}\label{sec:equiv-inv-tc}

In this section, we discuss the equivariant and invariant topological complexities of a wedge (Theorem~\ref{prop:tc-2cat-1}) and the quotient $X/A$ (Example~\ref{exam:equiv-inv-x-a}). For this purpose, we use Theorem~\ref{principal}. Additionally, we also consider the smash product $X\wedge S^k$ (Example~\ref{exam:equiv-inv-tc-smash}).  

\medskip Let $X$ be a $G$-space. Let $X^{[0,1]}$ denote the space of all continuous paths in $X$, equipped with the compact-open topology, and let $e_2^X:X^{[0,1]}\to X\times X$ be the fibration given by $e_2^X(\gamma)=\left(\gamma(0),\gamma(1)\right)$. Note that $X^{[0,1]}$ is a $G$-space via the pointwise action $(g\gamma)(t):=g\gamma(t)$, for any $g\in G$, $\gamma\in X^{[0,1]}$, and $t\in [0,1]$. The Cartesian product $X\times X$ is also a $G$-space via the diagonal action $g(x,y)=(gx,gy)$. Hence, $e_2^X:X^{[0,1]}\to X\times X$ is a $G$-map (in fact, it is a $G$-fibration, see \cite[p. 2308]{colman2013equivariant}). 

\medskip On the other hand, consider the space \[X^{[0,1]}\times_{X/G} X^{[0,1]}:=\{(\alpha,\gamma)\in X^{[0,1]}\times X^{[0,1]}:~G\alpha(1)=G\gamma(0)\}.\] Note that $G\times G$ acts on $X^{[0,1]}\times_{X/G} X^{[0,1]}$ via $(g,h)(\alpha,\gamma):=(g\alpha,h\gamma)$, and on $X\times X$ via $(g,h)(x,y)=(gx,hy)$. Hence, the map $p_X:X^{[0,1]}\times_{X/G} X^{[0,1]}\to X\times X$, given by $p_X(\alpha,\gamma):=(\alpha(0),\gamma(1))$, is a $G\times G$-map (indeed, it is a $G\times G$-fibration, see \cite[Proposition 3.7, p. 110]{lubawski2015}).

\begin{definition}[Equivariant and invariant topological complexities]\label{defn:equiv-inv-tc}
Let $X$ be a $G$-space.
\begin{enumerate}
    \item[(1)] \cite[Definition 5.1, p. 2309]{colman2013equivariant} The \textit{equivariant topological complexity} of $X$, denoted by $\text{TC}_G(X)$, is the smallest integer $k$ such that $X\times X$ can be covered by $G$-invariant open sets $U_1,\ldots,U_k\subset X\times X$, each which admitting a $G$-map $s_i:U_i\to X^{[0,1]}$ such that $e_2^X\circ s_i=\text{incl}_{U_i}$. If no such integer $k$ exists,  then we set $\text{TC}_G(X)=\infty$.
    \item[(2)] \cite[p. 110]{lubawski2015} The \textit{invariant topological complexity} of $X$, denoted by $\text{TC}^G(X)$, is the smallest integer $k$ such that $X\times X$ may be covered by $G\times G$-invariant open sets $U_1,\ldots,U_k\subset X\times X$, each which admitting a $G\times G$-map $s_i:U_i\to X^{[0,1]}\times_{X/G} X^{[0,1]}$ such that $p_X\circ s_i=\text{incl}_{U_i}$. If no such integer $k$ exists,  then we set $\text{TC}^G(X)=\infty$.
\end{enumerate}
\end{definition}

For convenience, we record the following standard properties.

\begin{lemma}\label{lem:equi-inv-prop}\rm{
    \noindent\begin{enumerate}
        \item[(1)] \cite[Theorem 5.2, p. 2309]{colman2013equivariant}, \cite[Proposition 3.25, p. 115]{lubawski2015} If $X$ $G$-dominates $Y$, then $\text{TC}_G(X)\geq \text{TC}_G(Y)$ and $\text{TC}^G(X)\geq \text{TC}^G(Y)$. In particular, if $X\simeq_G Y$, then $\text{TC}_G(X)=\text{TC}_G(Y)$ and $\text{TC}^G(X)=\text{TC}^G(Y)$.
        \item[(2)] \cite[Corollary 5.8, p. 2310]{colman2013equivariant}, \cite[Remark 3.24, p. 114]{lubawski2015} If $X$ is a $G$-connected $G$-CW-complex with $X^G\neq\varnothing$, then \[\max\{\text{TC}_G(X),\text{TC}^G(X)\}\leq 2\text{cat}_G(X)-1.\] 
        \item[(3)] \cite[Corollary 5.4(1), p. 2310]{colman2013equivariant}, \cite[Corollary 3.26, p. 115]{lubawski2015} If $X$ is a $G$-space, then \[\text{TC}(X^G)\leq\min\{\text{TC}_G(X),\text{TC}^G(X)\}.\]
    \end{enumerate} }
\end{lemma}

Theorem~\ref{prop:equiv-quotient} shows that the equality $\text{cat}_G(X/A)=\text{cat}_G(X)$ holds under the condition that the inclusion map $A\hookrightarrow X$ is $G$-homotopic to a constant map $\overline{x_0}:A\to X$ for some $x_0\in X^G$. We can obtain similar formulas for the equivariant and invariant topological complexities under the condition that $A$ is $G$-contractible to a fixed point $x_0\in X^G$. For this purpose, we show the equivariant version of \cite[Proposition 0.17, p. 15]{hatcher2002}. As a direct consequence, we compute the equivariant category, and the equivariant and invariant topological complexities of $X/A$. 

\begin{proposition}\label{prop:quaotient-map-equivalence}\rm{
    Let $X$ be a $G$-space and $A\subset X$ be an invariant subset such that the inclusion map $i:A\hookrightarrow X$ is a $G$-cofibration. If $A$ is $G$-contractible to a fixed point $x_0\in X^G$, then the quotient map $q:X\to X/A$ is a $G$-homotopy equivalence. In particular, we have  $\text{cat}_G(X/A)=\text{cat}_G(X)$,  $\text{TC}_G(X/A)=\text{TC}_G(X)$, and $\text{TC}^G(X/A)=\text{TC}^G(X)$.}  
\end{proposition}
\begin{proof}
    Let $H:A\times I\to A$ be a $G$-homotopy such that $H_0=1_A$ and $H_1=\overline{x_0}$. Since $i:A\hookrightarrow X$ is a $G$-cofibration, there exists a $G$-homotopy $\widehat{F}:X\times I\to X$ such that the following diagram
 \begin{eqnarray*}
\xymatrix{ A \ar@{^{(}->}[rr]^{\,\,i} \ar[dd]_{\iota_0}  &   &  X\ar[dd]^{1_X}\ar[dl]_{\iota_0}\\
& X\times [0,1] \ar@{-->}[dr]^{\,\,\widehat{F}}  &  \\
 A\times [0,1]\ar@{^{(}->}[ur]^{\,\,i\times 1_I}\ar[rr]_{F} &    & X}
\end{eqnarray*} is commutative, where $F=i\circ H$ and $\iota_0(x)=(x,0)$, for all $x$. Note that the map $\widehat{F}_1:X\to X$ satisfies $\widehat{F}_1(a)=x_0$,  for any $a\in A$. Then, we obtain a $G$-map $h:X/A\to X$ given by $h(q(x)):=\widehat{F}_1(x)$, for any $x\in X$. Hence, $h\circ q\simeq_G 1_X$. 

Furthermore, note that the $G$-homotopy $\widehat{F}:X\times [0,1]\to X$ satisfies $\widehat{F}(a,t)\in A$, for any $a\in A$ and any $t\in [0,1]$. Then, we obtain a $G$-homotopy $\varphi:X/A\times [0,1]\to X/A$ given by $\varphi(q(x),t):=q(\widehat{F}(x,t))$, for any $x\in X$ and $t\in [0,1]$. Note that $\varphi_0=1_{X/A}$ and $\varphi_1=q\circ h$, so $q\circ h\simeq_G 1_{X/A}$.

Therefore, $q:X\to X/A$ is a $G$-homotopy equivalence. By Proposition~\ref{prop-G-domina}, we conclude that the equality $\text{cat}_G(X)=\text{cat}_G(X/A)$ holds. Furthermore, by Lemma~\ref{lem:equi-inv-prop}(1), we have $\text{TC}_G(X/A)=\text{TC}_G(X)$ and $\text{TC}^G(X/A)=\text{TC}^G(X)$.
\end{proof}

 Given $G$-spaces $X$ and $Y$ with basepoints $x_0\in X^G$ and $y_0\in Y^G$, recall that the wedge $X\vee Y$ is a $G$-invariant and a $G\times G$-invariant subset of $X\times X$. Additionally, observe that $_{X\vee Y}\text{cat}_{G}(X\times Y)\leq _{X\vee Y}\text{cat}_{G\times G}(X\times Y)$. 

\medskip Now, we present the equivariant and invariant version of \cite[Theorem 3.6, p. 4371]{dranishnikov2014}. 

\begin{theorem}\label{thm:lower-bound-tc-wedge}\rm{
  Let $X$ and $Y$ be $G$-spaces and $x_0\in X^G$, $y_0\in Y^G$ be basepoints. Then, the wedge $X\vee Y$ satisfies the following inequalities:
  \[\max\{\text{TC}_G(X),\text{TC}_G(Y),\text{cat}_G(X\times Y)\}\leq \text{TC}_G(X\vee Y),\] and \[\max\{\text{TC}^G(X),\text{TC}^G(Y),_{X\vee Y}\text{cat}_{G\times G}(X\times Y)\}\leq \text{TC}^G(X\vee Y).\]}
\end{theorem}
\begin{proof}
    Recall that there are $G$-retractions ($G\times G$-retractions, respectively) $r_X:X\vee Y\to X$ and $r_Y:X\vee Y\to Y$ (see the proof of Proposition~\ref{prop:lower-bound-equiv-wedge}). Then, by Lemma~\ref{lem:equi-inv-prop}(1), we have \[\max\{\text{TC}_G(X),\text{TC}_G(Y)\}\leq \text{TC}_G(X\vee Y)\] (and similarly, $\max\{\text{TC}^G(X),\text{TC}^G(Y)\}\leq \text{TC}^G(X\vee Y)$, respectively).

On the other hand, observe that \[
X\times Y\subset \left(X\vee Y\right)\times \left(X\vee Y\right).
\] \begin{enumerate}
    \item[(1)] Given a $G$-invariant open set $U\subset \left(X\vee Y\right)\times \left(X\vee Y\right)$ together with a $G$-map $s:U\to \left(X\vee Y\right)^{[0,1]}$ such that $e_2^{X\vee Y}\circ s(u)=u$, for all $u\in U$. Consider the $G$-map $H:U\times [0,1]\to X\vee Y$, defined by $H\left((a,b),t\right):=s(a,b)(t)$. Note that $H\left((a,b),0\right)=a$ and $H\left((a,b),1\right)=b$.

Now, let $V:=U\cap \left(X\times Y\right)$ be a $G$-invariant open subset of $X\times Y$. Define the $G$-map $F:V\times [0,1]\to X\times Y$ by \[
F\left((a,b),t\right):=\left(r_X\circ H\left((a,b),t\right),r_Y\circ H\left((a,b),1-t\right)\right), 
\] for all $(a,b)\in V$ and $t\in [0,1]$. Note that, for any $(a,b)\in V$, we have \[F\left((a,b),0\right)=(a,b) \quad \text{ and } \quad F\left((a,b),1\right)=x_0(=y_0).\] This shows that $V$ is an open $G$-categorical set of $X\times Y$, so we conclude that \[\text{TC}_G(X\vee Y)\geq \text{cat}_G(X\times Y).\]

Therefore, we obtain \[\max\{\text{TC}_G(X),\text{TC}_G(Y),\text{cat}_G(X\times Y)\}\leq \text{TC}_G(X\vee Y).\]
\item[(2)] Given a $G\times G$-invariant open set $U\subset \left(X\vee Y\right)\times \left(X\vee Y\right)$ together with a $G\times G$-map $s:U\to \left(X\vee Y\right)^{[0,1]}\times_{(X\vee Y)/G} \left(X\vee Y\right)^{[0,1]}$ such that $p_{X\vee Y}\circ s(u)=u$, for all $u\in U$, consider the coordinate maps $\sigma,\rho:U\to \left(X\vee Y\right)^{[0,1]}$, such that $s(u)=(\sigma(u),\rho(u))$, for any $u\in U$. The fact that $s$ is a $G\times G$-map implies that \[\sigma(ga,hb)=g\sigma(a,b) \quad \text{ and } \quad \rho(ga,hb)=h\rho(a,b),\] for any $(g,h)\in G\times G$ and $(a,b)\in U$. Define the maps $H, F:U\times [0,1]\to X\vee Y$ by \[H\left((a,b),t\right):=\sigma(a,b)(t) \quad \text{ and } \quad F\left((a,b),t\right):=\rho(a,b)(t).\] Note the following: \begin{align*}
   H\left((a,b),0\right)&=a,\\ 
   F\left((a,b),1\right)&=b,\\
   G\left( H\left((a,b),1\right)\right)&=G\left( F\left((a,b),0\right)\right).
\end{align*} 
 Observe that $V:=U\cap \left(X\times Y\right)$ is a $G\times G$-invariant open subset of $X\times Y$. Define the $G\times G$-map $L:V\times [0,1]\to X\times Y$ by \[
L\left((a,b),t\right):=\left(r_X\circ H\left((a,b),t\right),r_Y\circ F\left((a,b),1-t\right)\right),\] for all $(a,b)\in V$ and $t\in [0,1]$. Note that for any $(a,b)\in V$, we have \[L\left((a,b),0\right)=(a,b) \quad \text{ and } \quad L_1\left(V\right)\subset X\vee Y.\] This shows that $\text{TC}^G(X\vee Y)\geq _{X\vee Y}\text{cat}_{G\times G}(X\times Y)$.

Therefore, we conclude \[\max\{\text{TC}^G(X),\text{TC}^G(Y),_{X\vee Y}\text{cat}_{G\times G}(X\times Y)\}\leq \text{TC}^G(X\vee Y).\] 
\end{enumerate} 
\end{proof}

As another application of Theorem~\ref{principal} together with Theorem~\ref{thm:lower-bound-tc-wedge}, we obtain the following statement, which is the equivariant and invariant case of \cite[Proposição 4.7.53, p. 177]{zapata2022}.

\begin{theorem}\label{prop:tc-2cat-1}\rm{
 Let $X$ and $Y$ be $G$-connected $G$-CW-complexes such that $X^G\neq\varnothing$ and $Y^G\neq\varnothing$. Suppose that $\text{cat}_G(X)\geq \text{cat}_G(Y)$.
 \begin{enumerate}
     \item[(1)] If $\text{TC}_G(X)=2\text{cat}_G(X)-1$, then \[\text{TC}_G(X\vee Y)=\max\{\text{TC}_G(X),\text{TC}_G(Y),\text{cat}_G(X\times Y)\}=\text{TC}_G(X).\] Furthermore, $\text{TC}_G(X\vee Y)=2\text{cat}_G(X\vee Y)-1$.
     
      \item[(2)] If $\text{TC}^G(X)=2\text{cat}_G(X)-1$, then \[\text{TC}^G(X\vee Y)=\max\{\text{TC}^G(X),\text{TC}^G(Y),_{X\vee Y}\text{cat}_{G\times G}(X\times Y)\}=\text{TC}^G(X).\] Furthermore, $\text{TC}^G(X\vee Y)=2\text{cat}_G(X\vee Y)-1$.
 \end{enumerate} }   
\end{theorem}
\begin{proof}
  By Lemma~\ref{lem:equi-inv-prop}(2), we know that $\text{TC}_G(X\vee Y)\leq 2\text{cat}_G(X\vee Y)-1$. Then, by Theorem~\ref{principal}, we conclude that \[\text{TC}_G(X\vee Y)\leq 2\text{cat}_G(X)-1\] (using the hypotheses $\text{cat}_G(X)\geq \text{cat}_G(Y)$).  
  
  \begin{enumerate}
      \item[(1)] Since $\text{TC}_G(X)=2\text{cat}_G(X)-1$, by Theorem~\ref{thm:lower-bound-tc-wedge}, we obtain \[\text{TC}_G(X\vee Y)=\max\{\text{TC}_G(X),\text{TC}_G(Y),\text{cat}_G(X\times Y)\}=\text{TC}_G(X).\] Furthermore, the equality $\text{TC}_G(X\vee Y)=2\text{cat}_G(X\vee Y)-1$ holds.
      \item[(2)] Since $\text{TC}^G(X)=2\text{cat}_G(X)-1$, by Theorem~\ref{thm:lower-bound-tc-wedge}, we obtain \[\text{TC}^G(X\vee Y)=\max\{\text{TC}^G(X),\text{TC}^G(Y),_{X\vee Y}\text{cat}_{G\times G}(X\times Y)\}=\text{TC}^G(X).\] Furthermore, the equality $\text{TC}^G(X\vee Y)=2\text{cat}_G(X\vee Y)-1$ holds.
  \end{enumerate}
\end{proof}

Theorem~\ref{prop:tc-2cat-1} implies the following proposition.

\begin{proposition}\label{prop:equiv-inv-tc-wedges-x-x}\rm{
Let $X$ be a $G$-connected $G$-CW-complex such that $X^G\neq\varnothing$, and $k\geq 1$ be an integer.
 \begin{enumerate}
     \item[(1)] If $\text{TC}_G(X)=2\text{cat}_G(X)-1$, then \[\text{TC}_G(\underbrace{X\vee\cdots\vee X}_{k \text{ times }})=\text{TC}_G(X).\] 
      \item[(2)] If $\text{TC}^G(X)=2\text{cat}_G(X)-1$, then \[\text{TC}^G(\underbrace{X\vee\cdots\vee X}_{k \text{ times }})=\text{TC}^G(X).\] 
 \end{enumerate} }       
\end{proposition}

We present the following example. For completeness, let us first recall that for $n\geq 2$, $S^n$ is the $n$-sphere with the group $\mathbb{Z}_2$ acting by reflection, given by multiplication by $-1$ in the last coordinate. We have $\left(S^{n}\right)^{\mathbb{Z}_2}=S^{n-1}$, and $\text{TC}_{\mathbb{Z}_2}(S^n)=2\text{cat}_{\mathbb{Z}_2}(S^n)-1=3$, as shown in \cite[Example 5.9, p. 2311]{colman2013equivariant}. On the other hand, \[\text{TC}^{\mathbb{Z}_2}(S^n)=\begin{cases}
       2,& \hbox{ if $n$ is even,}\\
       3,& \hbox{ if $n$ is odd,}\\
   \end{cases}\] as shown in \cite[Example 4.2, p. 115]{lubawski2015}. Also, given integers $n_1,\ldots,n_k\geq 1$ with $k\geq 2$, we have  $\text{TC}\left(S^{n_1}\vee\cdots\vee S^{n_k}\right)=3$, as shown in \cite[Proposição 4.7.50, p. 176]{zapata2022}.

\begin{example}\label{exam:sphere}\rm{
Let $n_1,\ldots,n_k\geq 2$ be integers. \begin{enumerate}
    \item[(1)] By Theorem~\ref{prop:tc-2cat-1}(1), we obtain \[\text{TC}_{\mathbb{Z}_2}\left(S^{n_1}\vee\cdots\vee S^{n_k}\right)=3.\] 
    \item[(2)] Suppose that $k\geq 2$. By Lemma~\ref{lem:equi-inv-prop}(3), \[\text{TC}^{\mathbb{Z}_2}\left(S^{n_1}\vee\cdots\vee S^{n_k}\right)\geq \text{TC}\left(S^{n_1-1}\vee\cdots\vee S^{n_k-1}\right)=3\] (using the fact that $(X\vee Y)^G=X^G\vee Y^G$). On the other hand, by Lemma~\ref{lem:equi-inv-prop}(2) together with Theorem~\ref{principal}, we obtain \[\text{TC}^{\mathbb{Z}_2}\left(S^{n_1}\vee\cdots\vee S^{n_k}\right)\leq 2\text{cat}_{\mathbb{Z}_2}\left(S^{n_1}\vee\cdots\vee S^{n_k}\right)-1=3.\] Therefore, $\text{TC}^{\mathbb{Z}_2}\left(S^{n_1}\vee\cdots\vee S^{n_k}\right)=3$. 
\end{enumerate}}
\end{example}

\begin{example}\label{exam:x-sphere}\rm{
Let $X$ be a $\mathbb{Z}_2$-connected $\mathbb{Z}_2$-CW-complex such that $X^{\mathbb{Z}_2}\neq\varnothing$ and $\text{cat}_{\mathbb{Z}_2}(X)\geq 2$. Recall that $2=\text{cat}_{\mathbb{Z}_2}(S^n)$.
\begin{enumerate}
    \item[(1)] Suppose that $\text{TC}_{\mathbb{Z}_2}(X)=2\text{cat}_{\mathbb{Z}_2}(X)-1$. Then, by Theorem~\ref{prop:tc-2cat-1}(1), we obtain \[\text{TC}_{\mathbb{Z}_2}(X\vee S^n)=\max\{\text{TC}_{\mathbb{Z}_2}(X),\text{TC}_{\mathbb{Z}_2}(S^n),\text{cat}_{\mathbb{Z}_2}(X\times S^n)\}=\text{TC}_{\mathbb{Z}_2}(X).\]  
     \item[(2)] Suppose that $\text{TC}^{\mathbb{Z}_2}(X)=2\text{cat}_{\mathbb{Z}_2}(X)-1$. Then, by Theorem~\ref{prop:tc-2cat-1}(2), we obtain \[\text{TC}^{\mathbb{Z}_2}(X\vee S^n)=\max\{\text{TC}^{\mathbb{Z}_2}(X),\text{TC}^{\mathbb{Z}_2}(S^n),_{X\vee S^n}\text{cat}_{\mathbb{Z}_2\times \mathbb{Z}_2}(X\times S^n)\}=\text{TC}^{\mathbb{Z}_2}(X).\]  
\end{enumerate}}   
\end{example}

From the first part of Theorem~\ref{prop:equiv-quotient} together with Theorem~\ref{prop:tc-2cat-1}, we have the following example. Note that,  given a pair $(X,A)$ of $G$-CW-complexes, the inclusion map $A\hookrightarrow X$ is a $G$-cofibration. 

\begin{example}\label{exam:equiv-inv-x-a}\rm{
 Let $(X,A)$ be a pair of $G$-CW-complexes. If the inclusion map $i:A\hookrightarrow X$ is $G$-homotopic to a constant map $\overline{x_0}:A\to X$, for some $x_0\in X^G$, then, by Theorem~\ref{prop:equiv-quotient}, we have that $X/A$ is $G$-homotopic to the wedge $X\vee S A$, where $A\times\{0\}$ is the basepoint of $S A$ and $x_0$ is the base point of $X$. Hence, by Lemma~\ref{lem:equi-inv-prop}(1), we have $\text{TC}_G(X/A)=\text{TC}_G(X\vee S A)$, and $\text{TC}^G(X/A)=\text{TC}^G(X\vee S A)$. In addition, suppose that $X$ and $S A$ are $G$-connected such that $\text{cat}_G(X)\geq 2$. Recall that $\text{cat}_G(S A)\leq 2$ (see Example~\ref{exam:1-2}). By Theorem~\ref{prop:tc-2cat-1}, we obtain the following equalities:
 \begin{enumerate}
     \item[(1)] If $\text{TC}_G(X)=2\text{cat}_G(X)-1$, then \[\text{TC}_G(X/A)=\text{TC}_G(X).\] 
     \item[(2)] If $\text{TC}^G(X)=2\text{cat}_G(X)-1$, then \[\text{TC}^G(X/A)=\text{TC}^G(X).\] 
 \end{enumerate} }   
\end{example}

From the proof of Theorem~\ref{thm:equiv-smash}, we obtain the following example.

\begin{example}\label{exam:equiv-inv-tc-smash}\rm{
    Let $X$ be a pointed $G$-space, and let $S^k$ be the $k$-dimensional sphere ($k\geq 0$) with the trivial action of $G$. By the first part of the proof of Theorem~\ref{thm:equiv-smash}, there exists a $G$-equivariant homeomorphism between $X\wedge S^k$ and the iterated reduced suspension $\Sigma^k X$ for each $k\geq 0$. Here $\Sigma^0 X=X$. Hence, by Lemma~\ref{lem:equi-inv-prop}(1), we have \[\text{TC}_G(X\wedge S^k)=\text{TC}_G(\Sigma^k X) \quad \text{ and } \quad \text{TC}^G(X\wedge S^k)=\text{TC}^G(\Sigma^k X).\] Recall that  $\text{cat}_G(\Sigma^k X)\leq 2$ for any $k\geq 1$. Additionally, suppose that $X$ is a $G$-CW-complex. By Lemma~\ref{lem:equi-inv-prop}(2), we have \[\max\{\text{TC}_G(X\wedge S^k),\text{TC}^G(X\wedge S^k)\}\leq 3,\] for any $k\geq 1$. } 
\end{example}

We end this article with the following remark.

\begin{remark}\label{rem:conj}
    Motivated by Theorem~\ref{thm:lower-bound-tc-wedge} and Theorem~\ref{prop:tc-2cat-1}, we conjecture that the following statement holds: Given $G$-connected $G$-CW-complexes $X$ and $Y$ with $X^G\neq\varnothing$ and $Y^G\neq\varnothing$, then: 
    \begin{align*}
     \text{TC}_G(X\vee Y)&=\max\{\text{TC}_G(X),\text{TC}_G(Y),\text{cat}_G(X\times Y)\}, \\ 
     \text{TC}^G(X\vee Y)&=\max\{\text{TC}^G(X),\text{TC}^G(Y),_{X\vee Y}\text{cat}_{G\times G}(X\times Y)\}.
    \end{align*} 
\end{remark}

This conjecture, in the nonequivariant case, was proved in \cite[Theorem 6, p. 409]{dranishnikov2019} under more restrictive hypotheses.

\bibliographystyle{plain}

\end{document}